\documentclass[11pt]{article}
\usepackage[utf8]{inputenc}
\usepackage{amsmath, amssymb, amsthm}
\usepackage{hyperref}

\newtheorem{theorem}{Theorem}
\newtheorem{lemma}{Lemma}
\newtheorem{corollary}{Corollary}
\theoremstyle{remark}
\newtheorem{remark}{Remark}

\title{The Planar Case of Thomas' Positive Circuits Conjecture}
\author{Natan Katz \\
\small Fachbereich Mathematik, Johann Wolfgang Goethe Universität \\
\small D-60054 Frankfurt am Main, Germany \\
\small \texttt{natan.katz@luminaisec.com}}
\date{December 6, 2025}

\begin{document}

\maketitle

\begin{abstract}
The notion of circuit refers to a cyclic oriented influence between the elements of a dynamical system. There are two classes of circuit: positive and negative. R. Thomas conjectured that a necessary condition of multi stationarity is the existence of positive circuits. In this paper we use dynamical system tools and planar analysis to find conditions for which the conjecture holds for planar systems.
\end{abstract}

\section{Introduction}
In the late 1940's Delbrück suggested a new approach to modelling cell differentiation ([5, 6]). He proposed a model admitting multi stationarity which is generated by mutual inhibition. This model gave rise to the idea that the phenomenon of multi stationarity, which was already well known in physics, is related to a mutual influence between the variables of a system. In 1980 Rene Thomas ([3, 6], [11] [20]) generalized Delbrück's suggestion with the following conjecture: ``The presence of a positive circuit is a necessary although not sufficient condition for the existence of multi stationarity''. This conjecture had been stated for both logical, or discrete, and continuous time functions ([16]). The former has been completely studied ([11, 13]), while the latter has been the subject of many works ([3, 6], [11]-[20]) proving it for different cases and using it for biological applications such as immunology and memory ([6, 16, 19, 20]).

In this paper we restrict attention to planar continuous systems. In such systems the limit set cannot contain strange attractors. Studying the nature of the zeros and periodic orbits of such a system provides a complete picture of the system's flow. It is known from the theory of planar dynamical systems that if any isolated zero is either stable or surrounded by periodic orbit for which the exterior side is stable, then it is unique. Therefore we search in this paper for conditions of the stability of zeros in a generic planar system which contains no positive circuits in the entire plane.

\section{What are Circuits?}
Consider the dynamical system
\begin{equation}
\frac{dx}{dt}=f(x)
\end{equation}
where $x$ is an $n$-dimensional vector. The Jacobian matrix of the system is the $n \times n$ matrix with components
\begin{equation}
a_{ij}=\frac{\partial f_{i}}{\partial x_{j}} \quad i,j=1,2,\dots,n
\end{equation}
A non-vanishing entry implies that the variable $x_{j}$ exerts an influence on variable $x_{i}$; a positive entry indicates activation, while a negative indicates inhibition. This idea extends to a larger number of variables. Picking an entry $a_{i_{1}i_{2}}$ then $a_{i_{2}i_{3}}$ and repeating this process until we reach an entry of the form $a_{i_{j}i_{1}}$, where $j \le n$, we form a closed path from the index $i_{1}$ to itself. The set of variables that correspond to the set of indices in the path is called a circuit and the variables are its edges.

The sign of the product of all of the picked entries represents the self exertion of each variable on itself via the rest of the variables in the circuit ([6, 16]). For a positive product of the circuit edges the circuit is called a positive circuit and implies self activation, while for a negative product it is called a negative circuit and implies self inhibition. (Remark: The non-vanishing terms of the Jacobian determinant may represent either circuits of length $n$ or a union of smaller circuits. Therefore one need not identify the notation of permutation that appears above with the notation in the theory of determinants).

Experiments which led to Thomas' conjecture have shown that positive circuits are responsible for multi stationarity of systems, while negative circuits are responsible for stable or periodic behavior. The variables in a system which has a few stable steady states may use positive circuits as a decision mechanism between these states. If there are no positive circuits the variables keep the system in one nearly optimal state ([13, 21]).

Circuits are studied not only in graph theory but also in applied sciences in particular biology ([2, 12, 21]). In immunology combinations of positive and negative circuits arise in T cell model of help suppressor interaction ([7, 9]). In neurobiology the neuron cells generate synapse connections between themselves: a cell exerts an influence on its neighbor via these connections. The sign of such circuits is determined by the parity number of the cells. Such structures are also used in memory modeling ([2, 6]). Circuits may use in models of spatial structures and competition between different populations.

\section{Problem Description}
In this section we present an analytic description of the problem. Let the planar system:
\begin{equation}
\frac{dx}{dt}=f(x,y)
\end{equation}
\begin{equation}
\frac{dy}{dt}=g(x,y)
\end{equation}
$F$ denotes the vector field $(f, g)$. Recall that the conjecture claims that the existence of a positive circuit is a necessary condition for multiple steady states. We assume that there are no positive circuits in the entire plane. Hence $\forall(x,y) \in \mathbb{R}^{2}$ we have
\begin{equation}
f_{x}(x,y) \le 0, \quad g_{y}(x,y) \le 0
\end{equation}
\begin{equation}
g_{x}(x,y)f_{y}(x,y) \le 0
\end{equation}
The theory of planar dynamical systems implies that if any isolated zero is stable then it is unique. Thus for each isolated zero of (3--4) we search conditions for its stability. We assume that the right hand side functions of equations (3--4) are smooth enough and that they do not vanish on any open set (i.e. in any $\epsilon$ neighborhood in the plane there exists a point $(x, y)$ such that $F(x,y)$ do not vanish).

\section{Preliminary Study}
We now consider a class of r.h.s functions that satisfy Thomas' conjecture.

\begin{lemma}
Let $x_{1}, x_{2}$ s.t. $x_{1} < x_{2}$, Let $I=[x_{1},x_{2}]\times\{0\}$ interval. If there exist
\begin{equation}
f(I)\equiv0
\end{equation}
\begin{equation}
g(x_{1},0)=g(x_{2},0)=0
\end{equation}
Then $g(I)\equiv0$.
\end{lemma}

\begin{proof}
Assume $g_{x}$ is not identically zero on $I$ then (8) implies that $g_{x}$ changes its sign in $I$. Let $x_{4}\in(x_{1},x_{2})$ s.t.
\begin{equation}
g_{x}(x_{4},0)>0
\end{equation}
(6) implies that there exists $\delta>0$ s.t.
\begin{equation}
f(x_{4},y)\le0 \quad \forall y \text{ such that } 0<y<\delta
\end{equation}
Since we assume that for any open set there exists a point that $f$ does not vanish there exist $x_{3}\in(x_{1},x_{2})$, $0<y_{3}<\delta$ s.t.
\begin{equation}
f(x_{3},y_{3})>0
\end{equation}
W.l.o.g $x_{4}<x_{3}$ Combining (10) and (11) we obtain:
\begin{equation}
f(x_{4},y_{3})\le0<f(x_{3},y_{3})
\end{equation}
which contradicts (5).
\end{proof}

\begin{lemma}
Let $y_{1}, y_{2}$ s.t. $y_{1}<y_{2}$. If $f$ vanishes on interval $\{0\}\times[y_{1},y_{2}]$ and $g$ vanishes on its edges then $g$ vanishes on the entire interval.
\end{lemma}

\begin{proof}
This lemma is obtained immediately from (5).
\end{proof}

\begin{remark}
Lemma 1 and 2 are valid for replacing $f$ with $g$ and $x$ axis with $y$ axis.
\end{remark}

\begin{lemma}
Let the system (3--4) satisfy (5--6). Assume $a, b$ are two distinct zeros of $F$. If $f$ and $g$ do not change their signs in the entire plane, then there exists a curve $\gamma$ such that the following exist:
\begin{enumerate}
    \item $a \in \gamma$.
    \item $F(\gamma) \equiv 0$.
    \item There exists $c \in \gamma$ s.t. $d(c,b) < d(a,b)$, where $d$ is the Euclidean planar metric.
\end{enumerate}
\end{lemma}

\begin{proof}
Let $f, g$ be nonnegative in the entire plane. Let $a=(0,0)$ and $b=(r,w)$. Assume $r, w$ are positive. (5) and the positivity of $f$ and $g$ imply that $f(x_{0},y_{0})=0 \Rightarrow \forall x > x_{0}, f(x,y_{0})=0$ and $g(x_{0},y_{0})=0 \Rightarrow \forall y > y_{0}, g(x_{0},y)=0$. We denote:
\begin{align}
r^{1} &= [0,r]\times\{0\} \\
r^{2} &= \{r\}\times[0,w] \\
l^{1} &= \{0\}\times[0,w] \\
l^{2} &= [0,r]\times\{w\}  
 \end{align}
Since $f, g$ are nonnegative in the entire plane we obtain
\begin{equation}
f(r^{1})=g(l^{1})=0
\end{equation}
The positivity of $f$ and (13) imply
\begin{equation}
f_{y}(r^{1})=0
\end{equation}
If $f(l^{2})\equiv0$ then by (13) and Lemma 1 we obtain that $l^{1}\cup l^{2}$ is a path of fixed points hence we obtain the required. Assume that $f$ is not identically zero in $l^{2}$. (5) and the positivity of $f$ imply that there exists $\delta>0$ s.t.
\begin{equation}
f(x,w)>0 \quad \text{such that } 0\le x\le\delta
\end{equation}
We define a function $h$ in the interval $[0,\delta]$:
\begin{equation}
h(x)=\inf\{y \mid 0\le y\le w \wedge f_{y}(x,y)>0\}
\end{equation}
Positivity of $f$, the definition of $h$ and (15) imply:
\begin{equation}
0\le h(x)<w \quad \forall x \text{ such that } 0\le x\le\delta
\end{equation}
Hence $h$ is bounded in $[0,\delta)$. Positivity of $f$ and (5) imply that $h$ is monotone. Thus one can see that there exists $0<\delta_{1}<\delta$ s.t. $h$ is continuous in $[0, \delta_{1}]$. Let $\Gamma$ be the graph of $h$ in $[0, \delta_{1}]$. Positivity of $f$ and (14) imply:
\begin{equation}
f(\Gamma)\equiv0
\end{equation}
(13) implies
\begin{equation}
g(0,h(0))=0
\end{equation}
The definition of $h$ and (6) imply that
\begin{equation}
g_{x}(\Gamma)\le0
\end{equation}
Since $r, w$ are positive, $h$ is monotone, $g$ is positive and by (19) we obtain:
\begin{equation}
g(\Gamma)\equiv0
\end{equation}
Hence $\Gamma$ is the required path.
\end{proof}

\begin{lemma}
Let a system (3--4) satisfy (5--6). If $f_{y}$ does not change its sign in the entire plane, then between two distinct zeros of the system there exists another zero.
\end{lemma}

\begin{proof}
Let $a, b$ be zeros such that $a=(0,0)$ and $b=(x,y)$. If $y=0$ then (5) and the monotonicity of $g$ imply that the segment $[0,x]\times\{0\}$ is a line of zeros. In the same way we can show that for $x=0$ we obtain that the segment $\{0\}\times[0,y]$ is a line of zeros. Assume $x, y$ are positive. If $g_{x}$ is not positive in the entire plane then it vanishes in the entire rectangle $(0,0), (x,0), (x,y), (0,y)$. Let $0<x_{1}<x$, $0<y_{1}<y$, we obtain from (5) and (6) that
\begin{equation}
f(x_{1},0)<0
\end{equation}
\begin{equation}
f(0,y_{1})>0
\end{equation}
Let $l$ be the path from $(x_{1},0)$ to $(0, y_{1})$ via the first quadrant. (22--23) imply that there exists $c\in l$ s.t. $f(c)=0$ as required.
\end{proof}

\begin{theorem}
Let system (3--4) satisfy (5--6). Assume that one of the following exists:
\begin{enumerate}
    \item $f$ and $g$ do not change their signs in the entire plane.
    \item $f_{y}$ does not change its sign in the entire plane.
\end{enumerate}
If the system has an isolated zero, then it is unique.
\end{theorem}

\begin{proof}
Part 1 is an immediate result of Lemma 3. Part 2 is proved by Lemma 4.
\end{proof}

\begin{remark}
In many applications one has a qualitative knowledge about the right hand side functions such as their signs and their dependencies on the independent variables. However, often these functions are not explicitly given or known. Since it ensures a unique isolated zero for given qualitative behaviors, Theorem 1 may become useful in such cases.
\end{remark}

\section{Periodic Solutions}
The following section studies possible asymptotic behavior of periodic orbit solutions of system (3--4).

\begin{lemma}
If system (3--4) has a periodic solution, then it surrounds a divergence free domain.
\end{lemma}

\begin{proof}
Let $\Gamma$ be a $T$-periodic solution of (3--4). Green's theorem implies
\begin{equation}
\int_{x}\int_{y}\nabla\cdot F(x,y)dxdy=\oint_{\Gamma}fdy-gdx=\int_{0}^{T}(f\dot{y}-g\dot{x})dt=\int_{0}^{T}(fg-gf)dt\equiv0.
\end{equation}
On the other hand (5) implies that the divergence does not change its sign, so it must be identically zero in the surrounded domain.
\end{proof}

\begin{corollary}
If a periodic solution of (3--4) is a limit cycle, then it attracts (respectively repels) only exterior trajectories.
\end{corollary}

\begin{proof}
Let $\Gamma$ be a periodic solution of (3--4) which is a limit cycle. Lemma 5 implies that $\Gamma$ surrounds a divergence free domain $D$. The theory of ODE ([4, 10]) implies that there are no limit cycles in $D$. Thus $\Gamma$ is not properly contained in $D$ and there is no trajectory in $D$ that has $\Gamma$ as its $\omega$ (resp. $\alpha$) limit. Hence it is a limit cycle only of exterior trajectories in a neighborhood which is not divergence free.
\end{proof}

\begin{lemma}
If a periodic solution of (3--4) is a limit cycle then it is stable.
\end{lemma}

\begin{proof}
Let $\Gamma$ be a periodic solution of (3--4). Since it is periodic, it contains no zero of $F$, hence it is not approached by a zero sequence. Thus there is an exterior neighborhood of $\Gamma$ that contains no zero of $F$. Since it is a limit cycle, Corollary 1 implies that there is an exterior neighborhood of $\Gamma$ that trajectories approach it as $t$ increases (respectively decreases). (5) implies that in this neighborhood the total divergence is negative. Hence there exists a trajectory that approaches $\Gamma$ as $t$ increases. Thus ([4, 10]) all the trajectories approach $\Gamma$ as $t$ increases hence it is stable.
\end{proof}

\section{Fixed Points}
This section studies the possible asymptotic behaviors of fixed points. We assume that the fixed points are isolated and restrict our study to the cases where the Jacobian matrix at the fixed point contains at least one non-vanishing entry.

\begin{lemma}
Any hyperbolic fixed point of system (3--4) is stable.
\end{lemma}

\begin{proof}
(5--6) imply that the determinant of the Jacobian is positive for any hyperbolic fixed point. Therefore the stability of a fixed point is determined by the trace sign. (5) implies that for any hyperbolic fixed point the trace is negative. Hence the fixed point is stable.
\end{proof}

\subsection{Non Hyperbolic Fixed Points}
In this section we study the asymptotic behavior of non hyperbolic fixed points. For this purpose we first study the asymptotic behavior of focus and nodes of system (3--4).

\begin{lemma}
If the origin is a node or focus, then it has no divergence free neighborhood.
\end{lemma}

\begin{proof}
Assume it has an $\epsilon$ neighborhood which is divergence free. There exists a $\delta$ neighborhood such that the $\omega$ (resp. $\alpha$) limit of all the trajectories in the neighborhood is the origin. Let
\begin{equation}
\epsilon_{1}=\min\{\epsilon,\delta\}
\end{equation}
In the $\epsilon_{1}$-neighborhood of the origin the flux is preserved since it is divergence free. But all the trajectories go to 0 as $t$ increases (respectively decreases). Thus we have a contradiction. Hence the divergence cannot be zero.
\end{proof}

\begin{lemma}
If the origin is a node or focus of (3--4) then it is stable.
\end{lemma}

\begin{proof}
Let the origin be a node or a focus. There exists an $\epsilon$ neighborhood such that all its trajectories approach the origin as $t$ increases or as it decreases. (5) and Lemma 8 imply that the divergence in this neighborhood is negative. Hence the flux is inward, so the origin is stable.
\end{proof} 

We obtain that nodes and foci are stable. The next sections use this result to identify the stability for several non-hyperbolic cases. We assume that the origin is an isolated fixed point and that the right hand side functions are smooth enough.

\subsubsection{Zero Eigenvalue with Multiplicity One}
There are three possible normal forms of nonhyperbolic fixed points in which the Jacobian trace does not vanish (i.e., there exists exactly one eigenvalue that is zero).

\paragraph{Case 1}
\begin{equation}
\frac{dx}{dt}=ax+by+H(x,y), \quad a<0
\end{equation}
\begin{equation}
\frac{dy}{dt}=R(x,y)
\end{equation}

\begin{lemma}
Let the origin be an isolated fixed point of system (26--27). Then the origin is a node.
\end{lemma}

\begin{proof}
(5--6) imply that $R$ is a function of order $m$ where $m$ is odd and $m\ge3$. We apply the following time rescaling:
\begin{equation}
\overline{t}=at
\end{equation}
and obtain the system
\begin{equation}
\frac{dx}{d\overline{t}}=x+\frac{b}{a}y+\frac{H(x,y)}{a}
\end{equation}
\begin{equation}
\frac{dy}{d\overline{t}}=\frac{R(x,y)}{a}
\end{equation}
where $a<0$. The Implicit Function Theorem implies that near the origin we have:
\begin{equation}
x=-\frac{b}{a}y+l(y)
\end{equation}
where $l$ is a function of order $m\ge2.$ (5) implies that the lowest coefficient that does not vanish, and contains $y$ in (30) is positive, and the order of the term is odd. (6) Implies that the lowest coefficient of a term of the form $x^{i}$ that does not vanish differs in its sign from $b,$ and $i$ is odd. Hence substituting (31) in (30) we obtain that the lowest term that does not vanish is positive and has odd order term. This form provides a node ([1, 10]).
\end{proof}

\paragraph{Case 2}
\begin{equation}
\frac{dx}{dt}=ax+H(x,y)
\end{equation}
\begin{equation}
\frac{dy}{dt}=cx+R(x,y)
\end{equation}
where $a<0$. Here the Implicit Function Theorem and (6) imply that the lowest coefficient of a term of the form $y^{j}$ in $H$ differs in its sign from $c$, $j$ is odd and satisfies:
\begin{equation}
x=-\frac{f_{y^{j}}}{a}y^{j}+l(y)
\end{equation}
where $l$ is of order greater than $j$. At this stage we can state the following lemma:

\begin{lemma}
If the origin is an isolated fixed point of (32--33), then it is a node.
\end{lemma}

\begin{proof}
Substituting (34) in (33) and using the same methods as in Lemma 10.
\end{proof}

\paragraph{Case 3}
\begin{equation}
\frac{dx}{dt}=ax+H(x,y)
\end{equation}
\begin{equation}
\frac{dy}{dt}=R(x,y)
\end{equation}
where $a<0$. In Case 3 we can use the same method as in the previous lemmas, but only for the cases where $f_{y}, g_{x}$ achieve a minimum or a maximum at the origin. Otherwise the origin is not necessarily a node.

\begin{remark}
We assumed during the discussion above that $f_{y}$ and $g_{x}$ were not identically zero. If $f$ is a function only of $x$, then we cannot use the Implicit Function Theorem, but since $f$ is monotone at $x$, (5) implies $f$ vanishes only at $x=0$. Lemma 2 implies then that the fixed point is not isolated.
\end{remark}

The following theorem summarizes this section:

\begin{theorem}
If the origin is an isolated fixed point of system (3--4) that satisfies:
\begin{enumerate}
    \item Conditions (5--6)
    \item There exists exactly one zero eigenvalue.
    \item $f_{y}$ and $g_{x}$ achieve a non-degenerate maxima or minima at the origin.
\end{enumerate}
Then the origin is a node.
\end{theorem}

\subsubsection{Zero Eigenvalue with Multiplicity Two}
The case that the Jacobian matrix has a zero eigenvalue with multiplicity two, but does not vanish, is analyzed in this section.

\begin{lemma}
If the origin is a nonhyperbolic fixed point of (3--4) that satisfies:
\begin{enumerate}
    \item Conditions (5--6).
    \item The Jacobian matrix has two zero eigenvalues but it does not vanish.
    \item It is not surrounded by periodic orbits.
\end{enumerate}
Then the origin is a node or focus.
\end{lemma}

\begin{proof}
We can write the equations in a normal form ([1, 10]):
\begin{equation}
\frac{dx}{dt}=y+l(y)
\end{equation}
\begin{equation}
\frac{dy}{dt}=a_{k}x^{k}(1+h(x))+b_{n}x^{n}y(1+g(x))+y^{2}R(x,y)
\end{equation}
where $l(y)$ is of order $m\ge2,$ where the functions $g, h$ and $R$ are smooth enough, and $g$ and $h$ vanish on $x=0$. (5--6) imply that $k$ is odd, $a_{k}$ is negative and that $n$ is even. Thus we may obtain that the origin is one of the following ([1, 10]):
\begin{enumerate}
    \item Center.
    \item Focus.
    \item Node.
\end{enumerate}
Since we assume that there are no periodic solutions only 2 and 3 are possible.
\end{proof}

The following theorem summarizes the results of this section.

\begin{corollary}
If the origin is a fixed point of system (3--4) that satisfies:
\begin{enumerate}
    \item (5--6)
    \item There are two zero eigenvalues and the Jacobian matrix does not vanish.
\end{enumerate}
Then the origin is stable.
\end{corollary}

\begin{proof}
This follows from Lemmas 12 and 9.
\end{proof}

\subsection*{Acknowledgements}
I wish to thank Professor Peter Kloeden and Dr. Stefan Siegmund for their beneficial ideas which enhanced my knowledge. I also thank Dr. Eliezer Schohat for fruitful discussions.

\end{document}